\documentclass{amsart}
\usepackage[utf8]{inputenc}
\usepackage[dvipsnames]{xcolor}

\usepackage{amssymb}
\usepackage{tikz}
\usepackage{tikz-cd}
\clearpage{}
\DeclareMathOperator{\Comm}{Comm}
\DeclareMathOperator{\Aut}{Aut}

\DeclareMathOperator{\HMod}{\mathcal{E}}

\DeclareMathOperator{\id}{id}

\newcommand{\free}{\mathbb{F}}
\newcommand{\Z}{\ensuremath{\mathbb{Z}}}

\newcommand{\sol}[1]{\widehat{#1}}\clearpage{}

\usepackage[colorlinks=true]{hyperref}
\usepackage[msc-links]{amsrefs}
\usepackage[capitalise]{cleveref}

\newtheorem{theorem}{Theorem}
\newtheorem*{mtheorem}{Main Theorem}
\newtheorem{lemma}[theorem]{Lemma}
\newtheorem*{corollary}{Corollary}

\theoremstyle{definition}
\newtheorem{definition}[theorem]{Definition}

\theoremstyle{remark}

\title[Hall's universal group in $\Comm(\free)$]{Hall's universal group is a subgroup of the abstract commensurator of a free group}
\author{Edgar A. Bering IV}
\address{Department of Mathematics, Technion---Israel Institute of Technology, Haifa, Israel}
\email{bering@campus.technion.ac.il}
\author{Daniel Studenmund}
\address{Department of Mathematical Sciences, Binghamton University, Binghamton, New York}
\email{daniel@math.binghamton.edu}
\thanks{EB was supported by the Azrieli Foundation.}

\begin{document}

\begin{abstract}
P. Hall constructed a universal countable locally finite group $U$, determined up to isomorphism by two properties: every finite group $C$ is a subgroup of $U$, and every embedding
of $C$ into $U$ is conjugate in $U$. Every countable locally finite group is a subgroup of $U$. We prove that $U$ is a subgroup of the abstract commensurator of a finite-rank nonabelian free group.
\end{abstract}

\maketitle

\section{Introduction}

A group $G$ is \emph{locally finite} if every finitely generated subgroup is finite. In 1959, P. Hall constructed a  \emph{universal} countable locally finite group $U$~\cite{hall-universal}. Hall's group $U$ is the unique countable group such that every finite group embeds in $U$ and any two isomorphic finite subgroups of $U$ are conjugate in $U$. These properties imply that any countable locally finite group $L$ embeds in $U$~\cite{hall-universal}*{Lemma 4}.

Given a group $G$ the \emph{abstract commensurator} of $G$ is the collection of isomorphisms $\phi : H\to K$ between finite-index subgroups $H, K \leq G$, modulo agreement on a finite-index domain. If two groups $G$ and $H$ have a third group $K$ as a common finite-index subgroup then $G$ and $H$ are \emph{commensurable} and $\Comm(G) \cong \Comm(H)$.

Let $\free_k$ denote the free group of rank $k$. As free groups of any finite rank $k\geq 2$ are commensurable, we have $\Comm(\free_k)\cong \Comm(\free_2)$. We will refer to this group as $\Comm(\free)$. For all $k$ there is an embedding $\Aut(\free_k) < \Comm(\free)$~\cite{bartholdi-bogopolski}*{Lemma 2.3}. The permutations on a free basis provide an embedding $S_k < \Aut(\free_k)$, consequently $\Comm(\free)$ contains every finite group. Our main result implies a stronger result: any countable ascending chain of finite groups can be realized in $\Comm(\free)$.

\begin{mtheorem}
    Hall's universal group $U$ is a subgroup of $\Comm(\free)$.
\end{mtheorem} 

\begin{corollary}
    If $L$ is a countable locally finite group, then $L$ is a subgroup of $\Comm(\free)$.
\end{corollary}

The proof of the main theorem uses the realization of the abstract commensurator $\Comm(\free)$ as the homotopy equivalence group of a {\em full solenoid} $\sol{\Gamma}$, the inverse limit of all finite-sheeted covers of any finite graph $\Gamma$ with nonabelian fundamental group~\cite{bering-studenmund}.
Hall's universal group can be constructed as the direct limit of a sequence finite permutation groups $G_k$. To prove the main theorem, each $G_k$ is realized as a group of homotopy equivalences of a finite graph $\Gamma_k$, and $G_k$-equivariant graph morphisms $\Gamma_{k+1} \to \Gamma_k$ are constructed to induce homotopy equivalences $\sol{\Gamma_{k+1}}\to \sol{\Gamma_k}$. These identify each $G_k$ as a group of homotopy equivalences of $\sol{\Gamma_0}$, from which the theorem follows.

The group $\Comm(\free)$ has been studied for some time, but relatively little is known about its structure. For example, it is not known whether $\Comm(\free)$ is simple. A'Campo and Burger noted that $\Comm(\free)$ is not linear~\cite{acampo-burger}, and  Bartholdi and Bogopolski  proved $\Comm(\free)$ is not finitely generated~\cite{bartholdi-bogopolski}. Macedo\'{n}ska, Nekrashevych, and Sushchansky showed that the group of bireversible automatic permutations over any finite alphabet can be identified as a subgroup of $\Comm(\free)$ \cite{mns}.  Bou-Rabee and the second author describe a homomorphic image of the Baumslag-Solitar group $BS(2,3)$ in $\Comm(\free)$ that is not residually finite~\cite{bou-rabee-studenmund}.

Free groups and their related structures have been fruitfully studied by analogy with fundamental groups of closed surfaces of genus at least two. We note there are similarly few structural results known about the abstract commensurator $\Comm(\pi_1(\Sigma))$, where $\Sigma$ is a closed surface of genus at least two. $\Comm( \pi_1(\Sigma))$ is known to not be finitely generated~\cite{bartholdi-bogopolski} and not to be linear over any field~\cite{bou-rabee-studenmund}. However any finite subgroup of $\Comm(\pi_1(\Sigma))$ has a cyclic subgroup of index at most 2~\cite{bou-rabee-studenmund}*{Proposition 4}, so Hall's universal group is not a subgroup of $\Comm(\pi_1(\Sigma))$.

\section{The solenoid model of commensurations}

A \emph{graph} is a 1-dimensional CW-complex. We will refer to the 0-cells as vertices and 1-cells as edges.  

\begin{definition}
    Let $(\Gamma, \ast)$ be a pointed finite graph. The \emph{full solenoid} over $\Gamma$, denoted $\sol{\Gamma}$, is the pointed topological space obtained as the inverse limit of the system of all pointed finite-sheeted covers of $(\Gamma, \ast)$
    \[ (\sol{\Gamma},\ast) = \varprojlim_{\Lambda \stackrel{k:1}{\to} \Gamma} (\Lambda, \ast_\Lambda) \]
\end{definition}

Given a pointed topological space $(X,\ast)$ let $\HMod(X,\ast)$ be the group of homotopy classes of pointed homotopy equivalences of $X$. In our previous work we proved that the homotopy equivalences of a full solenoid over a finite aspherical CW-complex model the abstract commensurator of a fundamental group~\cite{bering-studenmund}. Specialized to finite graphs, this model gives the following isomorphism.

\begin{theorem}[\cite{bering-studenmund}*{Corollary 3}]\label{solenoid-theorem}
    If $(\Gamma, \ast)$ is a finite graph with more edges than vertices, then
    \[ \HMod(\sol{\Gamma},\ast) \cong \Comm(\free) \]
\end{theorem}

Any finite graph $\Gamma$ with more edges than vertices can be used to construct the solenoid in \cref{solenoid-theorem}. The next lemma provides a way to explicitly describe the different isomorphisms coming from \cref{solenoid-theorem}.

\begin{lemma} \label{solenoid-lemma}
    Let $\Gamma_1$ and $\Gamma_2$ be finite graphs and suppose $\phi:(\Gamma_1, v_1) \to (\Gamma_2, v_2)$ is either a homotopy equivalence or a finite-sheeted covering map. Then $\phi$ induces a pointed homotopy equivalence of solenoids $\sol \phi : (\sol{\Gamma_1},\ast) \to (\sol{\Gamma_2},\star)$.
\end{lemma} 
\begin{proof}
    First, in either case, $\phi$ induces a map of solenoids as follows. Note that $\phi_\ast : \pi_1(\Gamma_1, \ast)\to\pi_1(\Gamma_2,\ast)$ is injective and has finite-index image. Thus, each pointed finite-sheeted cover $(\Lambda,\ast) \to (\Gamma_2,\ast)$ determines a
    finite-sheeted cover $q_\Lambda : (\Gamma_\Lambda,\ast)\to (\Gamma_1,\ast)$ where $\pi_1(\Gamma_\Lambda,\ast) = \phi_\ast^{-1}(\pi_1(\Lambda, \ast))$. By construction the
    composition $\phi\circ q_\Lambda$ lifts uniquely to a map $\phi_\Lambda : (\Gamma_\Lambda, \ast) \to (\Lambda, \ast)$. For each finite-sheeted cover $(\Gamma_\Lambda, \ast) \to \Gamma_1$ there is a projection map $\rho_{\Gamma_\Lambda} : \sol{\Gamma_1} \to \Gamma_\Lambda$. The induced map $\sol{\phi} : \sol{\Gamma_1} \to \sol{\Gamma_2}$ is
    the inverse limit of the set of maps $\{ \phi_\Lambda\circ \rho_{\Gamma_\Lambda} \}$ indexed over the finite-sheeted covers of $\Gamma_2$. Moreover, $\sol{\phi}$ covers $\phi$, in the sense that $\rho_{\Gamma_2} \circ \sol{\phi} = \phi \circ \rho_{\Gamma_1}$.

    Now, suppose $\phi$ is a finite-sheeted covering map. In this case, via the lifting construction, we see $\phi$ embeds the system of pointed finite-sheeted covers of $(\Gamma_1, v_1)$ as a cofinal subsystem of the system of pointed finite-sheeted covers of $(\Gamma_2, v_2)$. It is a standard fact about inverse limits that in this case $\phi$ induces a homeomorphism $\sol\phi : \sol{\Gamma_1} \to \sol{\Gamma_2}$.
    
    Finally, suppose $\phi$ is a homotopy equivalence with homotopy inverse $\psi$. The projection $\rho_1 : \sol{\Gamma_1} \to \Gamma_1$ is a fiber bundle with totally disconnected fibers~\citelist{\cite{bering-studenmund}*{Theorem 1} \cite{mccord}*{Theorem 5.6}}, and therefore $\rho_1$ has unique homotopy lifting~\cite{spanier}*{\S2.2}. The composition of induced maps $\sol{\psi}\circ \sol{\phi}$ covers the composition $\psi\circ\phi$. By lifting the homotopy $\psi \circ \phi\sim \id_{\Gamma_1}$ we see $\sol{\psi}\circ\sol{\phi}\sim \id_{\sol{\Gamma_1}}$. Symmetrically, $\sol{\phi}\circ\sol{\psi} \sim \id_{\sol{\Gamma_2}}$, and we are done.
\end{proof}

\section{Proof of the main theorem}

For any finite group $G$, let $U(G)$ be the underlying set. Given any set $S$, let $R(S)$ be the graph (treated as a CW complex) with one vertex $v$ and an edge $e_s$ for each $s\in S$. For a group $G$, fix a disjoint point $\ast$ and let $R(G) = R(U(G)\sqcup \{\ast\})$. There is an action $G\curvearrowright U(G)$ defined by $g\cdot h = gh$, which induces an action by graph automorphisms $G\curvearrowright R(G)$. 

Given any group $G$, let $S(G)$ be the group of permutations of the set $U(G)$. The left-multiplication action induces an injective homomorphism $\ell:G\to S(G)$.

\begin{lemma} \label{construction-lemma}
    For any finite group $G$, there is a continuous function $\phi: R(S(G)) \to R(G)$ such that 
    \begin{enumerate}
    \item $\phi$ is the composition of a homotopy equivalence and a finite-sheeted covering map, and
    \item $\phi$ is $G$-equivariant, where $G$ acts on $R(S(G))$ via the inclusion $\ell:G\to S(G)$.
    \end{enumerate}
\end{lemma}

\begin{proof}
    Identify $\pi_1(R(G), v)$ with the free group $\free(A(G))$ on the generating set $A(G) = \left\{ a_x \mid x\in U(G) \sqcup \{\ast\} \right\}$, via the map that sends $a_x$ to the image of the edge $e_x$ in $R(G)$. Let $n = [S(G) : G]$. Let $p : \Gamma \to R(G)$ be the covering space corresponding to the kernel of the map $\free(A(G)) \to \Z/n\Z$ defined by $a_g\mapsto 0$ for $g\in G$ and $a_\ast \mapsto 1$. For each $0\leq k < n$, let $S_k \subset \free(A(G))$ be
    \[
    S_k = \left\{ a_\ast^k a_g a_\ast^{-k} \mid g\in G \right\}.
    \]
    Choosing a base vertex $v_0\in \Gamma$, identify $\pi_1(\Gamma, v_0) \leq \free(A(G))$ with the subgroup generated by $\left( \bigcup_{k=0}^{n-1} S_k \right) \cup \{ a_\ast^n\}$ via path lifting. 
    
    The induced action $G\curvearrowright \pi_1(R(G), v)$ permutes the generators of $\pi_1(\Gamma, v_0)$, hence the action of $G$ on $R(G)$ lifts to an action on $\Gamma$. This action admits a concrete description, which we use to label the edges: The action of the deck group has a single cyclically ordered orbit of vertices $v_0, v_1, \dotsc, v_{n-1}$ in $\Gamma$. There is a single edge connecting $v_i$ to $v_{i+1}$ indexed modulo $n$, which we label $a_\ast^{i+1}$.
    At each $v_k$, there is a lift $R_k$ of the subgraph $R(U(G))$; the edges of this lift are naturally identified with the set $S_k$, and labeled by the same. Then $G$ acts on $\Gamma$ fixing the vertex set, permuting the edges of $R_k$ by $g \cdot a_*^k a_h a_*^{-k} = a^k a_{gh} a_*^{-k}$, and fixing all $n$ other edges.
    
    Let $f : \Gamma \to R$ be a homotopy equivalence collapsing the maximal subtree with edges labeled $\{a_\ast^i\}_{i=1}^{n-1}$. The graph $R$ has one vertex and the edge set consists of $\left| S(G) \right|$ loops labeled by $\left( \bigcup_{k=0}^{n-1} S_k \right)$ and a single loop labeled $a_\ast^n$. The action $G\curvearrowright \Gamma$ fixes every collapsed edge, hence descends to an action $G\curvearrowright R$ permuting each edge set $S_k$ as above, and fixing the edge $a_\ast^n$.
    Let $g : R \to \Gamma$ be a $G$-equivariant homotopy inverse of $f$
    
    Identify $R$ with $R(S(G))$ as follows. Choose right coset representatives $c_0, \dotsc, c_{n-1}$ for $\ell(G) \leq S(G)$. Define a graph isomorphism $\psi : R\to R(S(G))$ identifying the unique vertex in each graph, identifying the edge labeled $a_{\ast}^n$ with $\ast$, and mapping the edge labeled $a_*^k h a_*^{-k}$ to the edge labeled by $h c_k$. By construction, $\psi$ is $G$-equivariant. It follows that the composition $p \circ g \circ \psi^{-1} : R(S(G)) \to R(G)$ is the desired continuous function.
\end{proof}

\begin{figure}[hbt]
\begin{center}
\begin{tikzpicture}[vertex/.style={circle, draw, fill=black,
                        inner sep=0pt, minimum width=4pt},
                    every node/.style = {font=\footnotesize},
                    edge/.style={distance=1.5cm},
                    semithick]
    \node[vertex] (G) {};
    
    \draw      (G) to [in=  5, out= 85, loop, edge] node[auto] {$0$} (G);
    \draw      (G) to [in= 95, out=175, loop, edge] node[auto] {$1$} (G);
    \draw      (G) to [in=185, out=265, loop, edge] node[auto] {$2$} (G);
    \draw[red] (G) to [in=275, out=355, loop, edge] node[auto] {$\ast$} (G);
    
    \draw[->] (2.5,0) to node[auto] {$\rho$} (1.3,0);
    \draw[->] (5,-1.5) to node[auto] {$\psi\circ f$} (5,-3);
    \draw[->] (3.4,-3.4) to node[auto] {$\phi$} (1.1, -1.1);
    
    \begin{scope}[xshift=4cm,node distance=2cm]
    \node[vertex, label=left:{$v_0$}] (V0) {};
    \node[vertex, right of=V0, label=right:{$v_1$}] (V1) {};
    
    \draw[red] (V0) to [bend right] node[below] {$a_\ast$} (V1);
    \draw[red] (V0) to [bend left] node[auto] {$a_\ast^2$} (V1);
    
    \draw (V0) to [in= 50, out=130, loop, edge] node[auto] {$a_0$} (V0);
    \draw (V0) to [in=140, out=220, loop, edge] node[auto] {$a_1$} (V0);
    \draw (V0) to [in=230, out=310, loop, edge] node[auto] {$a_2$} (V0);
    
    \draw (V1) to [in=230, out=310, loop, edge] node[auto] {$a_\ast a_0a_\ast^{-1}$} (V1);
    \draw (V1) to [in=310, out= 40, loop, edge] node[auto] {$a_\ast a_1a_\ast^{-1}$} (V1);
    \draw (V1) to [in= 50, out=130, loop, edge] node[auto] {$a_\ast a_2a_\ast^{-1}$} (V1);
    \end{scope}
    
    \begin{scope}[xshift=5cm,yshift=-5cm, edge/.style={distance=2cm}]
    \node[vertex] (SG) {};
    
    \draw[red] (SG) to [in= 75, out=115, loop, edge] node[auto] {$a_\ast^2 = \ast$} (SG);
    \draw      (SG) to [in=130, out=170, loop, edge] node[auto] {$()$} (SG);
    \draw      (SG) to [in=180, out=220, loop, edge] node[auto] {$(012)$} (SG);
    \draw      (SG) to [in=230, out=270, loop, edge] node[auto] {$(021)$} (SG);
    \draw      (SG) to [in=280, out=320, loop, edge] node[auto] {$()(01)$} (SG);
    \draw      (SG) to [in=330, out= 10, loop, edge] node[auto] {$(012)(01)$} (SG);
    \draw      (SG) to [in= 20, out= 60, loop, edge] node[auto] {$(021)(01)$} (SG);
    
    \end{scope}
    
\end{tikzpicture} \end{center}
\caption{An illustration of \cref{construction-lemma} applied to $G = \mathbb{Z}/3\mathbb{Z}$. Elements
of $S(G) = S_3$ are listed in cycle notation, so that $\ell(1) = (012)$. }
\end{figure}
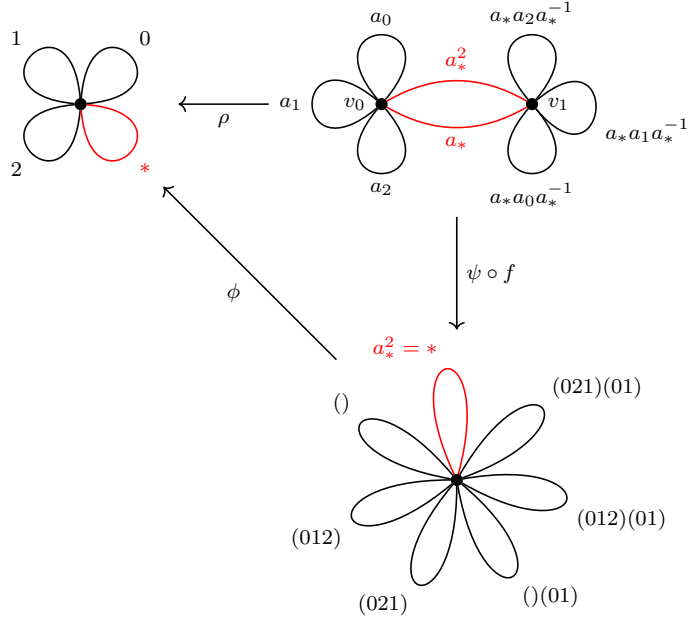

\begin{proof}[Proof of the main theorem]
We now prove that Hall's universal group $U$ is a subgroup of $\Comm(\free)$. Hall's universal group may be constructed as follows~\cite{hall-universal}*{\S1.2}: Fix any finite group $G_0$ with at least 3 elements and recursively define $G_{k+1} = S(G_k)$ for all $k\geq 0$. These form a directed sequence of groups under the left-multiplication maps $\ell_k : G_k \to G_{k+1}$. Hall's universal group $U$ is the colimit of this sequence.

Let $(G_k, \ell_k)$ be any directed sequence of groups as above. Let $R_0$ be the rose with two petals, and fix any covering space $p_0 : \Gamma_0 \to R_0$ such that $\pi_1(\Gamma_0)$ is free on $\left| G \right| + 1$ generators. Fix a homotopy equivalence $f_0 : R(G_0) \to \Gamma_0$. Define $i_0 = p_0 \circ f_0$. By Lemma \ref{solenoid-lemma}, $i_0$ induces a homotopy equivalence $\sol i_0 : \sol{R(G_0)} \to \sol{R_0}$.

For each $k$, let $\phi_k: R(G_k) \to R(G_{k-1})$ be a map satisfying the conditions of Lemma \ref{construction-lemma}. Then $\phi_k$ induces a pointed homotopy equivalence $\sol{R(G_k)} \to \sol{R(G_{k-1})}$, and therefore an isomorphism $\Phi_k : \HMod( \sol{R(G_k)}, \ast ) \to \HMod( \sol{R(G_{k-1})}, \ast )$.

Recursively define maps $i_k : R(G_k) \to R_0$ by $i_k = i_{k-1} \circ \phi_{k}^{-1}$. Each $i_k$ is a composition of coverings and homotopy equivalences by Lemma \ref{construction-lemma}, so induces a pointed homotopy equivalence $\sol{i_k} : \sol{R(G_k)} \to \sol{R_0}$ by Lemma  \ref{solenoid-lemma}.  Each action $G_k \curvearrowright R(G_k)$ induces an injective group morphism $G_k \to \HMod( \sol{R(G_k)}, \ast )$. Since each map $\phi_k$ is $G_k$-equivariant with respect to the left multiplication inclusion $\ell_k$, for each $k$ we have a commuting diagram
    \begin{center}
    \begin{tikzcd}
    G_k \ar{rr}{\ell_k} \ar[swap]{d} & & G_{k+1} \ar{d} \\
    \HMod( \sol{R(G_k)}, \ast ) \ar{dr} & & \HMod( \sol{R(G_{k+1})}, \ast ) \ar[swap]{ll}{\Phi_{k+1}} \ar{dl} \\
     & \HMod ( \sol{R_0}, \ast )
    \end{tikzcd}
    \end{center}

It follows that there is a map from the colimit of the directed sequence $(G_k, \ell_k)$, which is isomorphic to Hall's universal group $U$, to $\HMod ( \sol{R_0}, \ast )$. Because the vertical arrows are injective, this map is injective. We conclude $U < \Comm(\free)$ by \cref{solenoid-theorem}.
\end{proof}

\end{document}